\documentclass{article}
\usepackage{amssymb,amsmath,amsthm}
\usepackage{graphicx}
\usepackage{url}  
\usepackage{stmaryrd}  
\usepackage{eucal}  
\usepackage{enumerate}  
\usepackage[small,bf]{caption}
\usepackage{hyperref}
\usepackage{amsrefs}
\usepackage{multirow}
\usepackage[utf8]{inputenc}
\usepackage{mathptmx}

\newtheorem{thm}{Theorem}
\newtheorem{cor}[thm]{Corollary}

\newtheorem{pro}[thm]{Proposition}

\theoremstyle{df}
\newtheorem{df}[thm]{Definition}
\theoremstyle{remark}

\newtheorem{que}[thm]{Question}

\newcommand{\Nat}{\omega}

\newcommand{\Iff}{\Longleftrightarrow}
\renewcommand{\iff}{\leftrightarrow}

\newcommand{\dom}{\operatorname{dom}}
\newcommand{\mc}{\mathcal}
\newcommand{\eps}{\varepsilon}
\renewcommand{\phi}{\varphi}

\newcommand{\Hm}{\mathcal{H}}

\newcommand{\Cant}{2^\omega}
\newcommand{\Baire}{\omega^\omega}
\newcommand{\Str}{2^{<\omega}}
\newcommand{\Nstr}{\omega^{<\omega}}
\newcommand{\Cyl}[1]{\ensuremath{\llbracket #1 \rrbracket}}
\newcommand{\Rest}[1]{\ensuremath{|_{#1}}}
\newcommand{\Tup}[1]{\ensuremath{\langle #1 \rangle}}
\renewcommand{\S}{\ensuremath{\mathrm{\Sigma}}}
\renewcommand{\P}{\ensuremath{\mathrm{\Pi}}}

\newcommand{\Sleq}{\ensuremath{\subseteq}}

\newcommand{\Conc}{\ensuremath{\mbox{}^\frown}}
\newcommand{\Estr}{\ensuremath{\epsilon}}
\newcommand{\join}[1][\mbox{}]{\ensuremath{\oplus_{#1}}}
\newcommand{\T}{\ensuremath{\operatorname{T}}}

\begin{document}
	\title{Finding Subsets of Positive Measure\footnote{
		This is an extended journal version of the conference paper \cite{K.Reimann:10}.
		The final publication of \cite{K.Reimann:10} is available at Springer via http://dx.doi.org/10.1007/978-3-642-13962-8\_26.
	}}
	\author{Bj{\o}rn Kjos-Hanssen\footnote{This work was partially supported by
		a grant from the Simons Foundation (\#315188 to Bj\o rn Kjos-Hanssen) and by NSF grants DMS-0901020 and DMS-0652669.} \and Jan Reimann\footnote{
		Reimann was partially supported by Templeton Foundation Grant 13404 ``Randomness and the Infinite'' and by NSF grants DMS-0801270 and DMS-1201263.
		}} 
	\maketitle
	\begin{abstract}
		An important theorem of geometric measure theory (first proved by Besicovitch and Davies for Euclidean space) says that
		every analytic set of non-zero $s$-dimensional Hausdorff measure $\Hm^s$ contains a closed subset of non-zero (and indeed finite) $\Hm^s$-measure.
		We investigate the question how hard it is to find such a set, in terms of the index set complexity, and in terms of the complexity of the parameter needed to define such a closed set.
		Among other results, we show that given a (lightface) $\S^1_1$ set of reals in Cantor space, there is always a $\P^0_1(\mathcal{O})$ subset on non-zero $\Hm^s$-measure definable from Kleene's $\mc O$.
		On the other hand, there are $\P^0_2$ sets of reals where no hyperarithmetic real can define a closed subset of non-zero measure.
	\end{abstract}

	\section{Introduction}

		A most useful property of Lebesgue measure $\lambda$ is its (inner) \emph{regularity}: For any measurable set $E$, we can find an $F_\sigma$ set $F \subseteq E$ with $\lambda(E) = \lambda(F)$.
		In other words, any measurable set can be represented as an $F_\sigma$ set plus a nullset.
		This means that, for measure theoretic considerations, $E$ can be replaced by an $F_\sigma$, simplifying the complicated topological structure of arbitrary measurable sets.

		It is a basic result in geometric measure theory that regularity holds for  $s$-dimensional \emph{Hausdorff measure} $\Hm^s$ ($s>0$), too, with one important restriction. 

		\begin{thm}[Besicovitch and Moran \cite{BesicMoran}]
			If $E$ is $\Hm^{s}$-measurable and of finite $\Hm^{s}$-measure, then there exists an $F_\sigma$ set $F \subseteq E$ so that $\Hm^{s}(F) = \Hm^{s}(E)$. 
		\end{thm}

		In particular, one can approximate measurable sets of finite measure from inside through closed sets.
		\begin{cor}[Subsets of finite measure] \label{cor:inner_reg}
			\begin{equation} \label{equ:inner_reg}
				\Hm^{s}(E) = \sup \{ \Hm^{s}(C) \colon C \subseteq E \text{ closed}, \Hm^{s}(C) < \infty \},
			\end{equation}
		\end{cor}
 
		The requirement that $\Hm^{s}(E) < \infty$ is essential, since Besicovitch\cite{besicovitch:approximation_1954a} later showed that
		there is a $G_\delta$ set $G$ of Hausdorff dimension $1$ so that for every $F_\sigma$ subset $F \subset G$, the Hausdorff dimension of $G\setminus F$ is $1$, too. 

		Nevertheless, one can ask whether \eqref{equ:inner_reg} remains true even when $\Hm^{s}(E)$ is infinite. This is indeed so. In fact, one can approximate $E$ in measure by closed sets of \emph{finite measure}, provided $E$ is \emph{analytic}. 

		\begin{thm} \label{BD}
			The \emph{Subsets of finite measure} property \eqref{equ:inner_reg} holds for any analytic $(\pmb{\S}^1_1)$ subset of the real line.
		\end{thm}

		This result is one of the cornerstones of geometric measure theory, since it allows to pass from a set of infinite measure,
		which may be cumbersome to deal with since $\Hm^{s}$ is not $\sigma$-finite, to a closed set of finite measure, on which $\Hm^{s}$ is much better behaved.
		The theorem was first shown for closed sets in Euclidean space by Besicovitch\cite{Besicovitch} and extended to analytic sets by Davies \cite{Davies}. We will therefore refer to Theorem \ref{BD} also as the \emph{Besicovitch-Davies Theorem}.

		Besicovitch \cite{besicovitch:concentrated_1933a} had shown before that there exists a measurable set in the Euclidean plane every subset of which has $1$-dimensional measure either $0$ or $\infty$.
		Hence some restrictions on the definability of $E$ are necessary for \eqref{equ:inner_reg} to hold.

		Moreover, the existence of subsets of finite measure also depends on the underlying space, as well as on the nature of the dimension function.
		Davies and Rogers \cite{davies-rogers:problem_1969} constructed a compact metric space $X$ and a dimension function $h$ such that $X$ has infinite $\Hm^{h}$-measure but $X$ does not contain any sets of finite positive $\Hm^{h}$-measure.

		A few years before, on the other hand, Larman \cite{larman:hausdorff_1967} had shown that \eqref{equ:inner_reg} does hold for a class of compact metric spaces (those of \emph{finite dimension} in the sense of \cite{larman:theory_1967}).
		Rogers \cite{Rogers} proves it for complete, separable \emph{ultrametric spaces}. Hence \eqref{equ:inner_reg} holds for Cantor space $\Cant$ and Baire space $\Baire$.
		Most recently, using a quite different approach, Howroyd \cite{howroyd} was able to prove the validity of \eqref{equ:inner_reg} for any analytic subset of a complete separable metric space.
		It holds also for generalized Hausdorff measures $\Hm^{h}$, too, provided the dimension function $h$ does not decrease to $0$ too rapidly.

		In the following, we will study the complexity of finding subsets of positive measure in Cantor space $\Cant$, endowed with the standard metric
		\[
			d(x,y) = \begin{cases}
			   2^{-\min\{n\colon x(n) \neq y(n)\}} &  x \neq y \\
			   0 & x = y. \\
			\end{cases}
		\]
		The hierarchies of effective descriptive set theory allow for a further ramification of regularity properties. Any (boldface) Borel set is effectively (lightface) Borel relative to a parameter.
		Hence we can, for instance, given a (lightface) $\S^0_\alpha$ set, measure how hard it is to find a $\S^0_2(y)$ subset of the same measure, by proving lower bounds on the parameter $y \in \Cant$. 

		Dobrinen and Simpson~\cite{Dobrinen-Simpson} investigated this question for $\S^0_3$ sets in  Lebesgue measure and discovered an interesting connection with measure-theoretic \emph{domination properties}.
		Kjos-Hanssen \cite{Kjos-Hanssen:2007a} in turn linked measure-theoretic domination properties to LR-reducibility, a reducibility concept from algorithmic randomness.
		Recently, Simpson \cite{simpson} gave a complete characterization of the regularity problem for Borel sets with respect to Lebesgue measure.
		One of his results states that the property that every $\S^0_{\alpha+2}$ ($\alpha$ a recursive ordinal) subset of $\Cant$ has a $\S^0_2(Y)$ subset of the same Lebesgue measure holds if and only if $0^{(\alpha)} \leq_{LR} Y$.
		His paper \cite{simpson} also contains a survey of previous results along with an extensive bibliography.

		In this paper, we study the complexity of the corresponding inner regularity for Hausdorff measure on $\Cant$, extending and refining previous work by the authors \cite{K.Reimann:10}.
		We will see that, in contrast to the case of Lebesgue measure, finding subsets of positive Hausdorff measure can generally not be done with the help of a \emph{hyperarithmetic} oracle.
		The core observation is that determining whe\-ther a set of reals has \emph{positive Hausdorff measure} is more similar to determining whe\-ther it is \emph{non-empty} than to determining whe\-ther it has \emph{positive Lebesgue measure}.

		Determining the exact strength of the Besicovitch-Davies Theorem is not only of intrinsic interest.
		A family of important problems in theoretical computer science ask some version of the question to what extent randomness (which is a useful computational tool) can be extracted from a weakly random source (which is often all that is available).
		Such questions can also be expressed in computability theory. The advantage, and simultaneously the disadvantage, of doing so is that one abstracts away from considering any particular model of efficient computation. 
		One way to conceive of weak randomness is in terms of \emph{effective Hausdorff dimension}.
		Miller \cite{M} and Greenberg and Miller \cite{GM} obtained a negative result for randomness extraction: there is a real of effective Hausdorff dimension 1, that does not Turing compute any Martin-L\"of random real.
		Despite this negative result, effective Hausdorff dimension, which is a ``lightface'' form of Hausdorff dimension, has independent interest,
		as it seems to offer a way to redevelop much of geometric measure theory (for example Frostman's Lemma \cite{Reimann}) in a more effective way. 

		Another conception of weak randomness comes from considering sets that differ from Martin-L\"of random sets only on a sparse set of bits \cite{Extracting}, or sets that are subsets of Martin-L\"of sets \cites{Kurtz, MRL, Law}.
		Actually, these conceptions are related, as we will try to illustrate with the help of the set $\operatorname{BN1R}$ of all reals that bound no 1-random real in the Turing degrees,
		i.e., those reals to which no Martin-L\"of random real is Turing reducible.

		\begin{thm}\label{min}
			The set $\operatorname{BN1R}$ has Hausdorff dimension 1.
		\end{thm}
		\begin{proof}
			This is merely a relativization of a theorem of Greenberg and Miller \cite{GM}. 
		\end{proof}

		Theorem \ref{min} says that high effective Hausdorff dimension is not sufficient to be able to extract randomness. It can also be used to deduce that infinite subsets of random sets are not sufficiently close to being random, either.

		The set $\operatorname{BN1R}$ is Borel, so by the Besicovitch-Davies Theorem, for any $s < 1$ it has a closed subset $C$ that has non-zero $\Hm^s$-measure.
		Each closed set $C$ in Cantor space is $\P^{0}_{1}(x)$ for some oracle $x$. By a reasoning similar to \cite{DK}*{Theorem 4.3}, each $x$-random closed set contains a member of $C$.
		It follows by reasoning as in \cite{MRL} that each $x$-random set has an infinite subset that does not Turing compute any $1$-random (Martin-L\"of random) set.
		Thus, if $x$ could be chosen recursive, we would have a positive answer to the following question.

		\begin{que}\label{q1}
			Does each $1$-random subset of $\omega$ have an infinite subset that computes no $1$-random sets?
		\end{que}

		A partial answer to this question is known, using other methods:

		\begin{thm}[Kjos-Hanssen \cite{Law}]
			Each $2$-random set has an infinite subset that computes no $1$-random sets.
		\end{thm}

		But it is easy to see that the set $x$ just referred to cannot be chosen recursive. To wit, 
		any $\P^0_1$ class of non-zero $\Hm^s$-measure contains a $\P^0_1$ class consisting entirely of reals of effective Hausdorff dimension $\geq s$.
		By the computably enumerable degree basis theorem this subclass has a path of c.e.\ degree. But every such real, being of diagonally non-computable Turing degree, is Turing complete by Arslanov's completeness criterion,
		and hence computes a $1$-random\footnote{For these basic facts, on may consult a textbook such as Downey and Hirschfeldt \cite{DH}.}.

		In the present article we show that $X$ can be taken recursive in Kleene's $O$, but in general, for arbitrary $\S^{1}_{1}$ classes (or even just arbitrary $\P^{0}_{2}$ classes), $x$ cannot be taken hyperarithmetical. 
  
		We expect the reader to be familiar with  basic descriptive set theory and the effective part on hyperarithmetic sets and Kleene's $\mathcal{O}$.
		Standard references are \cite{Rogers} and \cite{Sacks}. We also assume basic knowledge of Hausdorff measures and dimension, as can be found in \cite{CARogers} or \cite{mattila}.

		\subsection*{\em The generalized join-operator}
			We will frequently need to generate sets of non-zero $\Hm^s$-measure. The following ``coded-product'' construction presents a convenient method to do this.

			Given two reals $x,y \in \Cant$ and an infinite, co-infinite $A \subseteq \Nat$, we define their \emph{$A$-join} $x \join[A] y$ as follows.
			Assume $A = \{ a_1 < a_2 < \dots \}$ and $\Nat\setminus A = \{b_1 < b_2 < \dots \}$. Let $x \join[A] y$ be the unique real $z$ such that $z(a_n) = x(n)$ and $z(b_n) = y(n)$ for all $n$. For sets 
			sets $X, Y \subseteq \Cant$, we define $X \join[A] Y$ as
			\[
				X \join[A] Y = \{ x \join[A] y \colon x \in X, y \in Y \}. 
			\]
			For rational $s = a/b$, $0 < s < 1$, $a,b$ relatively prime, the \emph{canonical $s$-join}, is given by letting $A = \{ bn + i \colon n \in \Nat, i < a  \}$. In this case, we write $x \join[s] y$ and $X \join[s] Y$.

			Measure-theoretically, the join behaves like a ``coded'' product.
 
			\begin{pro} \label{pro:hmeas-join}
				Assume $A \subseteq \Nat$ is such that  for some $r > 0$ and $c > 0$, 
				\[
					|A \cap \{0,\dots, n-1\}| \geq rn - c \quad \text{ for all $n$}. 
				\]	
				If $E, F \subseteq \Cant$, $s,t \in [0,1]$ are such that $\Hm^s E > 0$ and $\Hm^t F > 0$, then
				\[
					\Hm^{rs+(1-r)t} (E \join[A] F) > 0. 
				\]  
			\end{pro}
			This follows by a straightforward adaptation of the corresponding result for Euclidean spaces (see \cite{mattila}, Theorem 8.10).

	\section{The Besicovitch-Davies Theorem}
		In 1952, Besicovitch \cite{Besicovitch} proved the following theorem (for Euclidean space in place of $\Cant$).

		\begin{thm}
			If $F \subseteq \Cant$ is closed and $\Hm^{s} F = \infty$, then, for any $c > 0$ there exists a closed set $C \subseteq F$ such that $c < \Hm^s C < \infty$.
		\end{thm}

		The version for Cantor space follows from a paper by Larman \cite{Larman}. 
		Two technical lemmas play a crucial role in Besicovitch's proof (both hold in Cantor space, see e.g.\ \cite{CARogers}).

		\begin{enumerate}[(1)]
			\item The \emph{Increasing Sets Lemma} (valid in compact metric spaces): If $\{E_n\}$ is an increasing sequence of sets, then for $E = \bigcup E_n$, for any $m$,
			\[
				\Hm^s_{2^{-m}} E = \lim_n  \Hm^s_{2^{-m}} E_n.
			\] 
			(Note that here $\Hm^s$ is considered as an outer measure.)
	
			\item The \emph{Decreasing Sets Lemma}: If $\{C_n\}$ is a decreasing sequence of closed sets in $\Cant$, then for $C  = \bigcap C_n$,
			\[
				\Hm^s_{2^{-(m+1)}} E \geq \frac{c}{2} \lim_n \Hm^s_{2^{-m}} C_n,
			\]
			where $c$ is some positive, finite constant. (In Cantor space, we can choose $c = 1$)
	
		\end{enumerate}

		In the same journal in which Besicovitch's paper appeared, Davies \cite{Davies} published a proof showing that Besicovitch's result can be extended to analytic $(\mathbf{\S}^1_1)$ sets.

		We reformulate Davies' argument in Cantor space in a way suitable for our analysis.  

		\begin{thm} \label{Davies}
			Suppose $E \subseteq \Cant$ is $\S^1_1$. Assume further that $E$ is not $\sigma$-finite for $\Hm^s$, i.e.\ $E$ is not a countable union of measurable sets of finite $\Hm^s$-measure.
			Then there exists a closed set $C \subseteq E$ of infinite $\Hm^s$-measure.
		\end{thm}
		\begin{proof}	
				Pick a recursive relation $R(\sigma, \tau)$ such that 
			\[
				x \in E \quad \Iff \quad \exists g \in \Baire\: \forall n \in \Nat \: R(x\Rest{n}, g\Rest{n}).
			\]
			We define, for $\tau \in \Nstr$,
			\[
				A_\tau = \{x \colon \forall n \leq |\tau| \: R(x\Rest{n}, \tau\Rest{n})
			\]
			Then $\{A_\tau\}_{\tau \in \Nstr}$ forms a \emph{regular Souslin scheme} and we have
			\[
				E = \bigcup_{f \in \Baire} \bigcap_n A_{f\Rest{n}}. 
			\]
			Given $\alpha,\beta \in \Nstr\cup\Baire$, we write $\alpha \leq \beta$ if $\alpha(n) \leq \beta(n)$ for all $n \in \dom(\alpha) \cap \dom(\beta)$. Put 
			\[
				E^\sigma = \bigcup_{\substack{f \in \Baire \\ f \leq \sigma}} \bigcap_n A_{f\Rest{n}}.
			\]
			We have $E^{\Tup{n}} \nearrow E$. Choose $m_1$ so that $\Hm^s_{2^{-m_1}} E > 1$.  By the Increasing Sets Lemma, we can choose $r_1$ so that 
			\[
				\Hm^s_{2^{-m_1}} E^{\Tup{r_1}} > 1
			\]
			and $E^{\Tup{r_1}}$ is not $\sigma$-finite for $\Hm^s$, in particular $\Hm^s E^{\Tup{r_1}} = \infty$.
			The latter is possible since if there were no such $r_1$, as $E = \bigcup E^{\Tup{n}}$, $E$ would be $\sigma$-finite for $\Hm^s$, contradicting our assumption.

			Now we can continue the construction inductively. We obtain a function $r \in \Baire$  and a sequence of natural numbers $m_1 \leq m_2 \leq m_3 \leq \dots$ such that
			\begin{align*}
				 (a) & \qquad \Hm^s_{2^{-m_i}} E^{r\Rest n} > i \text{ for } 1\leq i \leq n, \\
				(b) & \qquad  E^{r\Rest n} \text{ is not $\sigma$-finite for $\Hm^s$}.
			\end{align*}

			We define 
			\[
				C_n = \bigcup_{\substack{|\tau| = n \\ \tau \leq r\Rest{n}}} A_{\tau} \qquad \text{and} \qquad C = \bigcap_n C_n.
			\]
			Note that each $C_n$, and hence $C$, is closed.
			By definition of $C_n$ we have $E^{r\Rest{n}} \subseteq C_n$ for all $n$. By $(a)$, $\Hm^s_{2^{-m_n}} C_n > n$. Moreover, $C_n \supseteq C_{n+1}$ for all $n$.
			We can hence apply the \emph{Decreasing Sets Lemma} and obtain $\Hm^s_{2^{-(m_n+1)}} C = \infty$ for all $n$, and thus $\Hm^s C = \infty$.

			It remains to show that $C \subseteq E$. Note that if $x \in C$, then for all $n$ there exists a $\tau_n \leq r\Rest{n}$ of length $n$ such that $x \in A_{\tau_n}$.
			The set of all such $\tau_n$ (for any $n$) forms an infinite, finite branching tree. Hence by König's Lemma, there exists $f \leq r$ such that $x \in \bigcap_n A_{f\Rest{n}}$, that is, $x \in E$.
		\end{proof}

		Note that, by the final argument of the preceding proof, we can write
		\[
			C = \bigcup_{g \leq r} \bigcap_n A_{g\Rest{n}} = \{ x \colon \exists g \leq r \, \forall n \: R(x\Rest{n}, g\Rest{n}) \}. 
		\]
		Note also that if $E$ is of $\sigma$-finite $\Hm^s$-measure, then the above construction may not produce, at some stage, an $r_n$ such that $E^{\Tup{r_1, \dots, r_n}}$ is of infinite $\Hm^s$-measure.
		But in this case $\Hm^s$ behaves like a finite Borel measure on a Borel set. We can then mimic the above construction, working directly with $\Hm^s$ instead of $\Hm^s_\delta$, and obtain a closed subset of positive $\Hm^s$-measure.

		Combining both cases, we obtain the following corollary.

		\begin{cor} \label{BD-cor}
			For each $\S^1_{1}$ class $E$ of non-zero $\Hm^s$-measure, written in canonical form
			\[
			   E = \{x \colon \exists g \,\forall n R(x\Rest{n},g\Rest{n})\}
			\]
			where $R$ is a recursive predicate, there exists a function $r\in \Baire$ such that for each $f$ majorizing $r$, the class
			\[
				C_{f}:=\{x \colon  \exists g \leq f \, \forall n \, R(x\Rest{n},g\Rest{n})\}
			\]
			is a $\P^0_1(f)$ subclass of non-zero $\Hm^s$-measure.
		\end{cor}

	\section{Index set complexity}

		In this section we determine the index set complexity of the following problem:
		\begin{quote}
			\em Given an index of an (effectively) analytic set $E \subseteq \Cant$, how hard is it to decide whether $E$ has non-zero (or finite) $\Hm^s$-measure?
		\end{quote}
		In our analysis we will always assume that $s$ is rational. This avoids technical complications arising from non-computable $s$ (which can be addressed by working relative to an oracle representing $s$).

		Initially, one may think that the computational difficulty in determining whe\-ther a set of reals has \emph{positive Hausdorff measure} could be similar to the difficulty in determining whether it has \emph{positive Lebesgue measure},
		but we find that it is more similar to the determining whether it is \emph{non-empty} -- and this is more difficult than the measure question.
		While questions about Lebesgue measure can often be answered using an arithmetical oracle, for non-emptiness we often have to go beyond even the hyperarithmetical.
		As we shall see, this level of difficulty first arises at the $G_{\delta}$ ($\mathbf{\P^{0}_{2}}$) level;
		we start by going over the simpler cases of open ($\mathbf{\S^{0}_{1}}$), closed ($\mathbf{\P^{0}_{1}}$), and $F_{\sigma}$ ($\mathbf{\S^{0}_{2}}$) sets.

		\begin{pro}\label{s01}
			For any rational $0 < s < 1$, the following families are identical, and have $\S^0_1$-complete index sets.
			\begin{enumerate}[(a)]
				\item $\S^0_1$ classes that are nonempty;
				\item $\S^0_1$ classes that have non-zero $s$-dimensional Hausdorff measure;
				\item $\S^0_1$ classes that have non-zero Lebesgue measure.
			\end{enumerate}
		\end{pro}
		\begin{proof}
			Given a set $W_e \subseteq \Str$, let $\Cyl{W} = \bigcup_{\sigma \in W} \Cyl{\sigma}$, where $\Cyl{\sigma} = \{x \in \Cant \colon \sigma \subset x \}$.
			Since any non-empty open set has positive Lebesgue measure, and having positive Lebesgue measure implies having infinite $\Hm^s$-measure for any $s< 1$, the three statements are equivalent.
			Any $\S^0_1$ class is given as $\Cyl{W_e}$ for some c.e.\ set $W_e$. The corresponding index sets are c.e.\ since
			$\Cyl{W_e} \neq \emptyset$ if and only if $W_e \neq \emptyset$ if and only if $\exists s,\sigma  \; \varphi_{e,s}(\sigma)\downarrow$, and they are complete by Rice's Theorem.
		\end{proof}

		Next, we compare the cases of $\P^{0}_{1}$ classes. It turns out deciding whether a $\P^0_1$ class has positive Lebesgue or Hausdorff measure is only slightly more complicated than deciding whether it is non-empty. 

		In the following, we let $T_e$ be the $e$-th recursive tree, 
		\[
			T_e = \{ \sigma \colon \forall \tau \Sleq \sigma \; \phi_{e,|\sigma|}(\tau) \uparrow \}.
		\]

		\begin{pro}\label{p01left}
			The set of indices of $\P^{0}_{1}$ classes that are nonempty is $\P^{0}_{1}$-complete.
		\end{pro}
		\begin{proof}
			A tree $T$ does not have an infinite path if and only if for some level $n$, no string of length $n$ is in $T$.
			If $T$ is recursive, the latter event is c.e.\ and hence the set $\{e \colon [T_e] \neq\emptyset \}$ is $\P^{0}_{1}$. It is  $\P^{0}_{1}$-hard by Rice's Theorem.
		\end{proof}

		\begin{pro}
			The set of indices of $\P^{0}_{1}$ classes that have positive Lebesgue measure is $\S^{0}_{2}$-complete.
		\end{pro}
		\begin{proof}
			Given a tree $T$, $[T]$ has positive Lebesgue measure if and only if $$\exists n \forall m \; |T^m | \geq 2^{m-n},$$ where $T^m = T \cap \{0,1\}^m$. This follows from the dominated convergence theorem.
			Hence the corresponding index set is $\S^0_2$.
			One can reduce the $\S^0_2$ complete set $\operatorname{Fin} = \{e \colon W_e \text{ finite} \}$ to it by effectively building, for each $e$, a tree $T_e$ such that if and only if a given $W_e$ is finite, the measure is positive.
			This is achieved by cutting the measure in half (i.e.\ terminating an appropriate number of nodes) whenever another number enters $W_e$.
			In detail, there exists (by the Church-Turing thesis) a recursive function $f$ such that $T^0_{f(e)} = \{\epsilon\}$ and 
			\[
				T^{s+1}_{f(e)} = \{\sigma\Conc i \colon \sigma \in T^s_{f(e)}, \: i \leq 1 \overset{{ }_\bullet}{-} |W_{e,s+1}\setminus W_{e,s}| \}.
			\]
			This $f$ is a many-one reduction from $\operatorname{Fin}$ to the set $\{e \colon [T_e] \text{ has positive Lebesgue measure}$
		\end{proof}

		\begin{thm}\label{p01}
			For any rational $0<s<1$, the set of indices of $\P^{0}_{1}$ classes of non-zero $\Hm^s$-measure is $\S^{0}_{2}$-complete. 
		\end{thm}
		\begin{proof}
			Given a tree $T$, $\Hm^s[T] = 0$ if and only if 
			\[
				\forall m \, \exists n \: \Hm^s_{2^{-n}} [T] < 2^{-m}. 
			\]
			By the Decreasing Sets Lemma, the latter is equivalent to
			\[
				\forall m \, \exists n,k \: \Hm^s_{2^{-n}} \Cyl{T^k} < 2^{-m}.
			\]
			The property $\Hm^s_{2^{-n}} \Cyl{T^k} < 2^{-m}$ however, is decidable:
			One has to check only a finite number of covers - in case $n \leq k$ any set $U \subseteq T \cap \Str$ such that $\Cyl{U} \supseteq \Cyl{T^m}$ and for all $\sigma \in U$, $n \leq |\sigma| \leq k$,
			and only the cover $\{\sigma \colon |\sigma| = n \text{ and } \sigma \text{ extends some } \tau \in T^k\}$ if $n > k$. It follows that the set
			\[
				\{e \colon [T_e] \text{ has non-zero $\Hm^s$-measure} \}
			\]
			is $\S^0_2$.

			We can again reduce $\operatorname{Fin}$ to this set to show it is $\S^0_2$-complete. This time the idea is to control the branching rate of a Cantor set.
			Whenever a new element enters $W_e$, we delay the next branching for a long time. Let $s = a/b$, $a,b$ relatively prime. 

			Define a recursive set $A \subseteq \Nat$ as follows: Set $l_0 = 0$ and $A\Rest{l_0} = \Estr$. Given $A\Rest{l_s}$, let
			\[
				l_{s+1} = \begin{cases}
					l_{s} + b & \text{if } W_{e,s+1}\setminus W_{e,s} = \emptyset, \\
					2^{l_{s}}       & \text{otherwise. } 
				\end{cases}	
			\]
			Put $A(i) = 1$ for all $l_s \leq i < a$ and $A(j) = 0$ for $a \leq j < l_{s+1}$. Finally define 
			\[
				C_{f(e)} = \Cant \join[A] \{0\},
			\]
			where $0$ denotes the real that is zero at all positions.

			If $W_e$ is infinite, then $A$ has large gaps, and it is not hard to see that in this case $C_{f(e)}$ has $\Hm^s$-measure $0$.
			If $W_e$ is finite, on the other hand, $C_{f(e)}$ is bi-Lipschitz equivalent to $\Cant \join[s] \{0\}$, and the latter set has positive $\Hm^s$-measure by Proposition \ref{pro:hmeas-join},
			a property that is preserved under bi-Lipschitz equivalence (see e.g.\ \cite{falconer:1990})
		\end{proof}

		Next, we look at the question whether a $\P^0_1$ class has \emph{finite} Hausdorff measure.
		It turns out this question is $\S^0_3$-complete, and hence indicates that finding closed subsets of finite measure is strictly more difficult in the case of Hausdorff measures than for Lebesgue measure.
		It is crucial here that Hausdorff measures are \emph{not} $\sigma$-finite.

		\begin{thm}
			For any rational $0<s<1$, the set of indices of $\P^{0}_{1}$ classes of finite $\Hm^s$-measure is $\S^{0}_{3}$-complete.
		\end{thm}

		\begin{proof}
			Given a tree $T$, $[T]$ has finite $\Hm^s$-measure if and only if
			\[
				\exists c \: \forall n \: \exists \text{ finite } F \subset T  \; \; [\text{ all $\sigma$ have length $\geq n$ and } \sum_{\sigma \in F} 2^{-s|\sigma|} < c ].
			\]	
			(It suffices to consider only finite covers since $[T]$ is compact.) Hence 
			\[
				\{e \colon [T_e] \text{ has finite $\Hm^s$-measure} \}
			\]
			is $\S^0_3$.
	
			We show it is $\S^0_3$-complete by reducing the set $\operatorname{Cof} = \{e \colon W_e \text{ is cofinite} \}$
			to it.
	
			Suppose $s = a/b$, where $a,b$ are relatively prime. Let $A = \{ bn + i \colon n \in \Nat, i < a  \}$. We define the co-r.e.\ set $B$ by letting $k \in B$ if and only if $k \in A$ or, if  there exists an $n$ such that $b(2n) \leq k < b(2n+1)$ and
			\[
				n \not \in W_{e}. 
			\]
			Put $C_{f(e)} = \Cant \join[B] \{0\}$. Since $B$ is co-r.e.\ it is straightforward to verify that $C_{f(e)}$ is $\P^0_1$.	
	
			We claim that $C_{f(e)}$ has finite $\Hm^s$-measure if and only if $W_e$ is cofinite.
			To see this, note that if $W_e$ is cofinite, $C_{f(e)}$ is bi-Lipschitz equivalent to $\Cant \join[s] \{0\}$, which has finite $\Hm^s$-measure (see, for example, \cite{mattila}).
			If, on the other hand, $W_e$ has an infinite complement, then there exist infinitely many blocks of size $b$ in $B$, as opposed to just the blocks of size $a$ a priori present in $B$.
			It follows that the Cantor-like set defined by $C_{f(e)}$ has finite $\Hm^h$-measure, where $\Hm^h$ is a generalized Hausdorff measure given by a dimension function
			\[
				h(2^{-n}) = 2^{-(sn + \alpha(n))},
			\]   
			where $\alpha(n) \to \infty$ for $n \to \infty$. It follows that $h(2^{-n})/2^{-sn} \to 0$ for $n \to \infty$. Therefore, $C_{f(n)}$ has infinite (in fact, non-$\sigma$ finite) $\Hm^s$-measure (see \cite{falconer:1990}, \cite{mattila}).
		\end{proof}

		\begin{pro}\label{s02left}
			The set of indices of $\S^{0}_{2}$ classes that are nonempty is $\S^{0}_{2}$-complete.      
		\end{pro}
		\begin{proof}
			Let $E$ be a $\S^0_2$ class, and let $R_e$ be a recursive relation so that
			\[
				x \in E \quad \iff \quad \exists n \, \forall m \; R_e(x\Rest{m}, n)
			\]
			$E$ is non-empty if and only if one of the $\P^0_1$ classes $\{z \colon \forall m \; R_e(z\Rest{m}, n) \}$ is non-empty.
			Deciding whether a $\P^0_1$-class is non-empty is $\P^0_1$-complete, as we saw in Proposition \ref{p01left}. Hence deciding whether, for given $e$, $\exists n \, \forall m \; R_e(x\Rest{m}, n)$, is $\S^0_2$-complete.
		\end{proof}

		\begin{pro}\label{s02}
			For any rational $0<s<1$, the set of indices of $\S^{0}_{2}$ classes of non-zero $\Hm^s$-measure is $\S^{0}_{2}$-complete.      
		\end{pro}
		\begin{proof}
			Assume $E = \bigcup_n F_n$, where each $F_n$ is closed. Let $E_n = \bigcup_{m\leq n} F_n$.  Since $E$ is measurable and $\Hm^s$ is a Borel measure, we have $\Hm^s E = \lim_n \Hm^s E_n$.
			Hence $E$ has non-zero $\Hm^s$-measure if and only if one of the $F_n$ has. 

			If $E$ is $\S^0_2$, then the indices of the $\P^0_1$ classes $E_n$ can be obtained effectively and uniformly. The result now follows from Theorem \ref{p01}.
		\end{proof}

		Passing from $\S^0_2$ to $\P^0_2$ classes, we see a significant jump in complexity.

		\begin{thm}\label{p02}
			For any rational $0<s<1$, the set of indices of $\P^{0}_{2}$ classes that have non-zero $\Hm^s$-measure is $\S^1_1$-complete. 
		\end{thm}         

		\begin{proof}
			Suppose $E$ is a $\P^{0}_{2}$ class. Consider the $s$-join
			\[
			 	F = \Cant \join[s] E  
			\]
			By Proposition \ref{pro:hmeas-join}, $F$ has non-zero $\Hm^s$-measure if and only if $E$ is not empty.
			Since the set of indices of $\P^{0}_{2}$ classes in $\Cant$ that are nonempty is $\S^{1}_{1}$-hard, so is the set of indices of $\P^{0}_{2}$ classes that have non-zero $\Hm^s$-measure. 
			By Corollary \ref{BD-cor}, the set of indices of $\S^1_1$ classes that are of non-zero $\Hm^s$-measure is $\S^1_1$, since 
			\[
			\Hm^s \{x \colon \exists g \forall n \, R(x\Rest{n},g\Rest{n})\} > 0 \quad 
			 \Iff \quad \exists f \;  \Hm^s C_f > 0,
			\]
			where $C_f$ is the $\P^0_1(f)$ class from Corollary \ref{BD-cor}.
		\end{proof}

		A straightforward computation shows that the set of indices of $\P^0_2$ classes that have non-zero Lebesgue measure is $\S^0_3$.
		Hence at the level $\P^0_2$ it is far more complicated to determine whether a class has non-zero Hausdorff measure than whether it has non-zero Lebesgue measure. 

		Our results are summarized in Figure \ref{summary-table}.

		\begin{figure}
			\begin{center}
				\begin{tabular}{|c|c|c|c|}
					\hline
					Family      & Nonempty? & Positive Hausdorff measure?  & Positive Lebesgue measure? \\
					\hline
					&&&\\
					$\S^0_1$  &  $\S^0_1$-complete  &  $\S^0_1$-complete  & $\S^0_1$-complete  \\
					&&&\\
					\hline
					&&&\\
					$\P^0_1$  &  $\P^0_1$-complete  &   & \\
					&&$\S^0_2$-complete & $\S^0_2$-complete\\
					$\S^0_2$  &  $\S^0_2$-complete &    &    \\
					&&&\\
					\hline
					&&&\\
					  $\P^0_2$  &   &  & $\S^0_3$ \\
					  & $\S^1_1$-complete &  $\S^1_1$-complete &\\
					$\S^1_1$  &    &    & \\
					&&&\\
					\hline
				\end{tabular}
			\end{center}
			\caption{Index set complexity of some classes of reals. For example, the set of indices of $\P^{0}_{2}$ classes that are of non-zero Hausdorff measure is $\S^{1}_{1}$-complete, and this is shown in Theorem \ref{p02}.}
			\label{summary-table}
		\end{figure}

	\section{Closed subsets of non-zero Hausdorff measure}\label{3}
		We now turn to the question how difficult it is to find a closed subset of non-zero Hausdorff measure.
		We will measure this in terms of the recursion theoretic complexity of the parameter needed to define such a closed subset.

		Corollary \ref{BD-cor} tells us that given a $\S^1_1$ class $E$ of non-zero $\Hm^s$-measure, we can find a function $r: \Nat \to \Nat$ such that there exists a $\P^0_1(r)$ subclass of non-zero $\Hm^s$-measure. How complex is $r$?
		We will see that a few fundamental results in higher recursion theory facilitate the classification of the possible complexities.

		\begin{df}
			A set $B \subseteq \Baire$ is called a \emph{basis} for a pointclass $\Gamma$ if each nonempty collection of reals that belongs to $\Gamma$ has a member in $B$.
		\end{df}

		We shall be particularly interested in the case $\Gamma=\S^{1}_{1}$. Here several bases are known.

		\begin{thm}[Basis theorems for $\S^{1}_{1}$]\label{basis}
			Each of the following classes is a basis for $\S^{1}_{1}$:
			\begin{enumerate}
				\item[$(1)$] $\{x \colon x \le_{\T} \mc O\}$, the reals recursive in some $\P^{1}_{1}$ set $($Kleene, see Rogers \cite{Rogers}*{XLII(b)}$)$;
				\item[$(2)$] $\{x \colon x<_{h}\mc O\}$, the reals of hyperdegree strictly below $\mc O$ $($Gandy \cite{Gandy}; see also Rogers \cite{Rogers}*{XLIII(a)}$)$;
				\item[$(3)$] $\{x \colon x\not\le_{h} z \And z\not\le_{h} x\}$, where $z$ is any given non-hyperarithmetical real $($Gandy, Kreisel, and Tait \cite{GKT}$)$.
			\end{enumerate}
		\end{thm}

		We first show that any basis for $\S^1_1$ contains a function that specifies a subset of non-zero Hausdorff measure.

		\begin{thm}\label{oh}
			Let $0< s < 1$ be rational. For each set $B \subseteq \Baire$ that is a basis for $\S^{1}_{1}$ and each $\S^{1}_{1}$ class $E$ of non-zero $\Hm^s$-measure,
			there is some $f \in B$ such that $E$ has a $\P^{0}_{1}(f)$ subclass of non-zero $\Hm^s$-measure.
		\end{thm}
		\begin{proof}
			Let $R$ be a recursive predicate such that 
			\[
				E =  \{x \colon \exists g \forall n \, R(x\Rest{n},g\Rest{n})\}
			\]
			For any function $f: \Nat \to \Nat$, 
			\[
				C_{f} =\{x \colon  \exists g \leq f \, \forall n \, R(x\Rest{n},g\Rest{n})\}
			\]  
			is a $\P^{0}_{1}(f)$ subclass of $E$.
			Now consider the set 
			\[
			\{ f\in\Baire \colon \Hm^s C_{f} > 0 \}
			\]
			By Theorem \ref{p01}, this is a $\S^{0}_{2}$ class in $\Baire$, in particular it is $\S^{1}_{1}$, hence it has a member $f \in B$. $C_{f}$ for such an $f$ is a $\P^{0}_{1}(G_{f})$ class, where $G_f$ is the graph of $f$.
		\end{proof}

		In particular, $E$ always has a $\P^{0}_{1}(\mc O)$ subclass of non-zero Hausdorff measure. 
		We will see next that there are examples, even of $\P^0_2$ classes, where no hyperarithmetical real is powerful enough to define a $\P^0_1$ subclass of non-zero Hausdorff measure. 

		\begin{thm}\label{top}
			Let $0< s < 1$ be rational. There is a $\P^{0}_{2}$ class $G$ of non-zero $\Hm^s$-measure such that the following holds:
			If $x \in \Cant$ is such that some $\P^0_1(x)$ subclass of $G$ has non-zero $\Hm^s$-measure, then $x \geq_{\T} H$ for every $H \in \operatorname{HYP}$.
		\end{thm}
		\begin{proof}
			Let $\operatorname{HYP}$ denote the collection of all hyperarithmetical reals. Note that the set
			\[
				E = \{ z \in \Cant \colon \forall H \in\operatorname{HYP}\, H \le_{\T} z \}
			\]
			is $\S^{1}_{1}$ (an observation made by Enderton and Putnam \cite{EP}).
			$E$ has Hausdorff dimension $1$, since it contains the upper cone of $\mc O$, and Reimann \cite{reimann:phd} has shown that the upper cone of any Turing degree has Hausdorff dimension $1$. It follows that $\Hm^t E = \infty$ for any $t < 1$.  

			Suppose $x$ is such that there is a $\P^{0}_{1}(x)$ subclass of $E$ that is of non-zero $\Hm^s$-measure and hence non-empty.
			We apply two basis theorems for $\P^0_1$ classes (or rather, their relativized versions).
			By the low basis theorems each $H \in \operatorname{HYP}$ is recursive in a real $y$ that is low relative to $x$,
			and by the hyperimmune-free basis theorem, each $H \in \operatorname{HYP}$ is recursive in a real $z$ that is hyperimmune-free relative to $x$. 

			The reals $y$ and $z$ form a minimal pair over $x$, since no non-computable degree comparable with $x'$ can be hyperimmune-free relative to $x$, and being hyperimmune-free relative to $x$ is closed downwards in the Turing degrees.
			Hence we must have that  $H\le_{T} x$ for every $H \in \operatorname{HYP}$.

			It remains to show that we can replace $E$ by a $\P^0_2$ class with the same property. Every $\S^1_1$ class is the projection of a $\P^0_2$ class.
			Instead of using the standard projection on coded pairs, we can use a projection along a $t$-join, i.e.\  there exists a $\P^0_2$ class $G \subseteq \Cant$ such that
			\[
				E = \{ z \colon \exists y \; (z \join[t] y\in G) \}. 
			\]
			with $t = s+\eps < 1$, where $\eps >0$ is sufficiently small.

			If we let $\pi_t(z \join[t] y) = z$ be the projection of an $t$-join onto the first ``coordinate'', then for all $x \in G$,
			\[
				d(\pi_t(x), \pi_t(x')) \leq d(x,x')^t,
			\]
			hence $\pi_t$ is Hölder continuous with exponent $t$. It follows that 
			\[
				\Hm^{s} G \geq \Hm^{s/t} \pi_s(G) = \Hm^{s/t} E = \infty. 
			\]
			On the other hand, every element of $G$ still computes every hyperarithmetic real, since every element of $G$ is the join of an element of $E$ with another real.
			Hence the argument above remains valid and we get that if $x$ defines a $\P^0_1(x)$ subclass of $G$ of non-zero $\Hm^s$-measure, $x \geq_{\T} H$ for every $H \in \operatorname{HYP}$.
		\end{proof}

		We can also give a sufficient condition for hyperarithmeticity based on the ability to define a closed subset of positive measure. This follows from Solovay's characterization of hyperarithmetic reals through fast growing functions.

		\begin{df}[Solovay \cite{Solovay}]
			A family $F$ of infinite sets of natural numbers is said to be \emph{dense} if each infinite set of natural numbers has a subset in $F$.
			A set $A$ of natural numbers is said to be \emph{recursively encodable} if the family of infinite sets in which $A$ is recursive is dense. 
		\end{df}

		\begin{thm}[Solovay \cite{Solovay}]\label{Solo}
			The recursively encodable sets coincide with the hyperarithmetic sets.
		\end{thm}

		\begin{thm}\label{hyponly}
			Let $0< s < 1$ be rational, assume that $E \subseteq \Cant$ is $\S^1_1$, and let $y \in \Cant$. If for some $U$ then $Y$ is hyperarithmetical.
		\end{thm}
		\begin{proof} 
			Suppose $y$ is recursive in each $x$ defining a $\P^0_1$ subclass of non-zero $\Hm^s$-measure of $E$. In
			particular, it is recursive in any (graph of a) function $r$ as in Corollary \ref{BD-cor}, and any $f$ dominating $r$. If
			$A \subseteq \Nat$ is infinite, it has an infinite subset $B = \{b_0 < b_1 < b_2 < \dots \}$ so that the function
			$p_B(n) = b_n$ dominates $r$ and hence defines a closed subset of non-zero $\Hm^s$-measure. It follows that $y$ is recursively encodable and thus hyperarithemtic.
		\end{proof}

	\section{Mass problems}\label{4}
		It is beneficial to phrase the preceding results as \emph{mass problems}. Recall that a mass problem is a subset of $\Cant$. Given a $\S^1_1$ class $E$ of non-zero $\Hm^s$-measure, $0 < s < 1$ rational, we define the mass problem
		\[
			S(E) = \{ x \in \Cant \colon \text{$E$ has a $\P^{0}_{1}(x)$ subclass of non-zero $\Hm^s$-measure} \}. 
		\]
		For sets of reals $X$, $Y$, $X$ is called \emph{weakly (Muchnik) reducible} to $Y$, $X \leq_w Y$ if for each $y \in Y$ there is some $x \in X$ such that $x \leq_{\T} y$.

		Our results now read as follows.
		\begin{enumerate}[(1)]
			\item If $z \in \Cant$ is $\P^{1}_{1}$-complete then $S(E)\le_{w}\{z\}$ for any $E$. (Theorem \ref{oh})
			\item There is a $\P^{0}_{2}$ class $G$ such that for each hyperarithmetical real $y$, $\{y\}\le_{w} S(G)$. (Theorem \ref{top})
			\item For each real $y$, if $\{y\}\le_{w}S(E)$ for some $E$, then $y$ is hyperarithmetical. (Theorem \ref{hyponly})
		\end{enumerate}
		The situation is summarized in Figure \ref{fig:Ln}.
		\begin{figure}[htb!]
		\centering%
		\includegraphics[scale=.33]{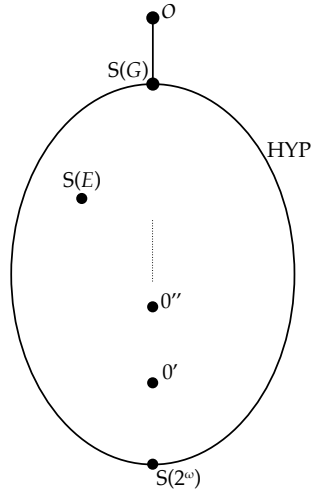}
		\caption{
			The relative position in the Muchnik lattice of the various mass problems $S(E)$.
			At the top is Kleene's $\mathcal{O}$, according to Theorems \ref{basis}(1) and \ref{oh}. The ellipse represents the hyperarithmetical sets $\operatorname{HYP}$ with their cofinal sequence $\{0^{(\alpha)} \colon \alpha < \omega_1^{CK} \}$.
			The top $S(G)$ class is located as indicated in Theorem \ref{top}. Each $S(E)$ bounds only sets in HYP, per Theorem \ref{hyponly}.
			It is not known which of the classes $E$ here displayed might represent the set of Turing degrees in $\operatorname{BN1R}$.
		}
		\label{fig:Ln}
		\end{figure}
	
\end{document}